\documentclass[10pt]{amsart}

\usepackage{mathrsfs}

\newtheorem{theorem}{Theorem}[section]
\newtheorem{lemma}[theorem]{Lemma}
\newtheorem{proposition}[theorem]{Proposition}

\calclayout
\title[Counting paths in perfect trees]{Counting paths in perfect trees}

\author{Peter J. Humphries}
\address{Department of Mathematics and Physics, North Carolina Central University, Durham, N.C., U.S.A., and Department of Mathematics and Statistics, University of Canterbury, Christchurch, New Zealand}
\email{pjhumphries@gmail.com}



\date{\today}

\begin{document}

\begin{abstract}
We present some exact expressions for the number of paths of a given length in a perfect $m$-ary tree. We first count the paths in perfect rooted $m$-ary trees and then use the results to determine the number of paths in perfect unrooted $m$-ary trees, extending a known result for binary trees.
\end{abstract}

\maketitle

\section{Introduction}\label{intro}

A tree $T=(V,E)$ is a connected acyclic graph with a finite vertex set $V$ and finite edge set $E\subseteq\binom{V}{2}$. The distance $d_T(u,v)$ between two vertices $u,v\in V$ is the number of edges in the (unique) path in $T$ that joins $u$ and $v$. In this paper, we focus on counting the pairs of vertices that are some given distance apart, or equivalently the paths of a given length, in a perfect tree.

Given a tree $T$, let $P(T,t)$ denote the number of pairs of vertices at distance exactly $t\ge 1$ from each other. That is,
\[
P(T,t)=\left|\left\{\{u,v\}\in\binom{V}{2}:d_T(u,v)=t\right\}\right|\quad{\rm and}\quad \sum_{t\ge 1}P(T,t)=\frac{|V|\left(|V|-1\right)}{2}.
\]
Note immediately that $P(T,1)=|E|$. Furthermore, from the observations that each vertex $v$ of degree $\deg(v)$ is the central vertex of $\binom{\deg(v)}{2}$ distinct paths of length $2$ and that each edge $\{u,v\}$ is the central edge of $(\deg(u)-1)(\deg(v)-1)$ distinct paths of length $3$, we obtain
\[
P(T,2)=\sum_{v\in V}\binom{\deg(v)}{2}\quad{\rm and}\quad P(T,3)=\sum_{\{u,v\}\in E}\left(\deg(u)-1\right)\left(\deg(v)-1\right).
\]
Similar expressions for $P(T,t)$ when $t\ge 4$ become increasingly complicated.

Faudree et al.~\cite{fau73} constructed examples showing that two non-isomorphic trees $T_1,T_2$ can have identical path length distributions (that is, $P(T_1,t)=P(T_2,t)$ for all $t$). Tight upper and lower bounds for $P(T,t)$ were given by Dankelmann~\cite{dan11} in terms of $|V|$ and either the radius or diameter of $T$.

A binary tree $T$, in which every vertex has degree $1$ or degree $3$, is perfect (or balanced) if $T$ has the maximum number of vertices among all binary trees of the same diameter. De Jong et al.~\cite{dej15} used a recursive approach to show that the perfect binary tree $T$ of diameter $D$ with $n$ degree-$1$ vertices has
\begin{align*}
P(T,t)=
\begin{cases}
2^{\frac{t+1}{2}}\left(n-3\cdot 2^{\frac{t-3}{2}}\right), &t\ {\rm odd},\\
3\cdot 2^{\frac{t}{2}-1}\left(n-2^{\frac{t}{2}}\right),&t\ {\rm even}
\end{cases}
\end{align*}
paths of length $t$ for $3\le t\le D$.

We adopt a different approach to extend this to perfect $m$-ary trees, where each vertex has degree $1$ or $m+1$. In particular, we prove the following theorem.

\begin{theorem}\label{unrootedmary}
Let $T$ be the perfect unrooted $m$-ary tree of diameter $D$. Then, for $1\le t\le D$,
\begin{align*}
P(T,t)&=
\begin{cases}
m^{\frac{t-1}{2}}\left(V(D)-V(t-1)\right),&t\ {\rm odd},\\
\frac{1}{2}(m+1)m^{\frac{t}{2}-1}\left(V(D)-V(t-1)\right),&t\ {\rm even},
\end{cases}
\end{align*}
where $V(d)$ is the number of vertices in the perfect unrooted $m$-ary tree of diameter $d$.
\end{theorem}

We first derive an analogous theorem for perfect rooted $m$-ary trees, where the root has degree $m$ and all other vertices have degree $1$ or $m+1$. Theorem~\ref{rootedmary} is obtained in Section~\ref{rooted} by counting the paths of length $t$ in a perfect rooted $m$-ary tree according to minimum depth, considering odd $t$ and even $t$ separately.

\begin{theorem}\label{rootedmary}
Let $T$ be the perfect rooted $m$-ary tree of depth $r$, and let $t$ satisfy $1\le t\le 2r$. If $t$ is odd, then
\begin{align*}
P(T,t)&=
\begin{cases}
m^{\frac{t-1}{2}}\left(V_R(r)-V_R(\frac{t-1}{2})\right)-\frac{t-1}{2}m^{t-1}, &t\le r,\\
m^{\frac{t-1}{2}}\left(V_R(r)-V_R(\frac{t-1}{2})\right)-\left(r-\frac{t-1}{2}\right)m^{t-1}, &t>r,
\end{cases}
\end{align*}
and if $t$ is even, then
\begin{align*}
P(T,t)&=
\begin{cases}
\frac{1}{2}(m+1)m^{\frac{t}{2}-1}\left(V_R(r)-V_R(\frac{t}{2}-1)\right)-\frac{t}{2}m^{t-1}, &t\le r,\\
\frac{1}{2}(m+1)m^{\frac{t}{2}-1}\left(V_R(r)-V_R(\frac{t}{2}-1)\right)-\left(r-\frac{t}{2}+1\right)m^{t-1}, &t>r,
\end{cases}
\end{align*}
where $V_R(d)$ is the number of vertices in the perfect rooted $m$-ary tree of depth $d$.
\end{theorem}

In Section~\ref{unrooted}, we use the results from Section~\ref{rooted} to prove Theorem~\ref{unrootedmary}.

\section{Perfect rooted $m$-ary trees}\label{rooted}

In a rooted $m$-ary tree $T=(V,E)$, there is a distinguished vertex $\rho$ of degree $m$ called the root, while every other vertex has degree $1$ or $m+1$. The depth $r$ of $T$ is the maximum value of $d_T(\rho,v)$ over all vertices $v\in V$. We call $T$ perfect if and only if every degree-$1$ vertex is distance $r$ from the root $\rho$.

Let $T$ be the perfect rooted $m$-ary tree of depth $r$. For $0\le s\le r$, there are exactly $m^s$ vertices $v\in V$ for which $d_T(\rho,v)=s$. Let $p=v_0\cdots v_{t}$ be a path of length $t$ in $T$. Then there is a unique vertex $v_s$, $0\le s\le t/2$, such that $d_T(\rho,v_s)\le d_T(\rho,v_i)$ for all $0\le i\le t$. We call $p$ a type-$[s,t-s]$ path rooted at $v_s$.

\begin{lemma}\label{pathcountbytypemary}
Let $T$ be the perfect rooted $m$-ary tree of depth $r$. If $r<t-s$, then the number of type-$[s,t-s]$ paths in $T$ rooted at $\rho$ is $0$. If $r\ge t-s$, then the number of type-$[s,t-s]$ paths in $T$ rooted at $\rho$ is
\begin{align*}
\begin{cases}
m^t,&s=0,\\
(m-1)m^{t-1},&0<s<\frac{t}{2},\\
\frac{1}{2}(m-1)m^{t-1},&s=\frac{t}{2}.
\end{cases}
\end{align*}
\end{lemma}

\begin{proof}
The case $r<t-s$ is obvious, as is the case $r\ge t-s$ with $s=0$. Assume that $r\ge t-s$ and that $0<s<\frac{t}{2}$. Then any type-$[s,t-s]$ path can be decomposed into a type-$[0,s]$ path rooted at $\rho$ and a type-$[0,t-s]$ path rooted at $\rho$, where these two paths are disjoint. There are $m^{s}$ choices for the type-$[0,s]$ path. Once this choice has been made, there are $(m-1)m^{t-s-1}$ choices for the type-$[0,t-s]$ path, so the total number of type-$[s,t-s]$ paths rooted at $\rho$ is $2\binom{m}{2}m^{t-2}=(m-1)m^{t-1}$. If $s=\frac{t}{2}$, then this argument counts each type-$[s,s]$ path twice, hence the third equality.
\end{proof}

Let $P_m(r,t)$ denote the number of paths of length $t$ in the perfect rooted $m$-ary tree of depth $r$. The preceding lemma can be used to derive exact expressions for $P_m(r,t)$. We consider paths of odd length and paths of even length separately, and make repeated use of the identity
\begin{align*}
\sum_{i=a}^{b}m^i=\frac{m^{b+1}-m^a}{m-1}.
\end{align*}

\begin{proposition}\label{oddpathsmary}
The number of paths of length $t=2k-1$ in the perfect rooted $m$-ary tree of depth $r$, where $1\le k\le r$, is
\begin{align*}
P_m(r,t)&=
\begin{cases}
m^{2k-2}\left(\frac{m^{r-k+2}-1}{m-1}-(r-k+2)\right),&r<2k-1,\\
m^{2k-2}\left(\frac{m^{r-k+2}-1}{m-1}-k\right),&r\ge 2k-1.
\end{cases}
\end{align*}
\end{proposition}

\begin{proof}
Let $T$ be the perfect rooted $m$-ary tree of depth $r$. If $r<k$, then the longest path in $T$ has length $2r<t$, and so $P_m(r,t)=0$.

If $k\le r<2k-1$ and $2k-r-1\le s\le k-1$, then by Lemma~\ref{pathcountbytypemary}, there are $(m-1)m^{t-1}$ type-$[s,t-s]$ paths rooted at each vertex $v$ for which $d_T(\rho,v)\le r-t+s$. Therefore,
\begin{align*}
P_m(r,t)&=(m-1)m^{2k-2}\sum_{s=2k-r-1}^{k-1}\left(\sum_{d=0}^{r-2k+s-1}m^d\right)\\
&=(m-1)m^{2k-2}\sum_{s=2k-r-1}^{k-1}\frac{m^{r-2k+s+2}-1}{m-1}\\
&=m^{2k-2}\sum_{i=1}^{r-k+1}\left(m^i-1\right)\\
&=m^{2k-2}\left(\frac{m^{r-k+2}-1}{m-1}-(r-k+2)\right).
\end{align*}

If $r\ge 2k-1$, then there are $m^{t}$ type-$[0,t]$ paths rooted at each vertex $v$ for which $d_T(\rho,v)\le r-t$. Furthermore, for $1\le s\le k-1$, there are $(m-1)m^{t-1}$ type-$[s,t-s]$ paths rooted at each vertex $v$ for which $d_T(\rho,v)\le r-t+s$. Therefore,
\begin{align*}
P_m(r,t)&=m^{2k-1}\sum_{d=0}^{r-2k+1}m^d+(m-1)m^{2k-2}\sum_{s=1}^{k-1}\left(\sum_{d=0}^{r-2k+s+1}m^d\right)\\
&=m^{2k-2}\left(m\left(\frac{m^{r-2k+2}-1}{m-1}\right)+(m-1)\sum_{s=1}^{k-1}\frac{m^{r-2k+s+2}-1}{m-1}\right)\\
&=m^{2k-2}\left(\frac{m^{r-2k+3}-m}{m-1}+\sum_{i=r-2k+3}^{r-k+1}\left(m^i-1\right)\right)\\
&=m^{2k-2}\left(\frac{m^{r-2k+3}-m}{m-1}+\frac{m^{r-k+2}-m^{r-2k+3}}{m-1}-(k-1)\right)\\
&=m^{2k-2}\left(\frac{m^{r-k+2}-1}{m-1}-k\right).
\end{align*}
\end{proof}

\begin{proposition}\label{evenpathsmary}
The number of paths of length $t=2k$ in the perfect rooted $m$-ary tree of depth $r$, where $1\le k\le r$, is
\begin{align*}
P_m(r,t)&=
\begin{cases}
m^{2k-1}\left(\frac{1}{2}(m+1)\left(\frac{m^{r-k+1}-1}{m-1}\right)-(r-k+1)\right),&r<2k,\\
m^{2k-1}\left(\frac{1}{2}(m+1)\left(\frac{m^{r-k+1}-1}{m-1}\right)-k\right),&r\ge 2k.
\end{cases}
\end{align*}
\end{proposition}

\begin{proof}
Let $T$ be the perfect rooted $m$-ary tree of depth $r$. If $r<k$, then the longest path in $T$ has length $2r<t$, and so $P(d,t)=0$.

If $k\le r<2k$, then by Lemma~\ref{pathcountbytypemary}, there are $\frac{1}{2}(m-1)m^{t-1}$ type-$[k,k]$ paths rooted at each vertex $v$ for which $d_T(\rho,v)\le r-k$. Furthermore, for $2k-r\le s\le k-1$, there are $(m-1)m^{t-1}$ type-$[s,t-s]$ paths rooted at each vertex $v$ for which $d_T(\rho,v)\le r-t+s$. Therefore,
\begin{align*}
P_m(r,t)&=\frac{1}{2}(m-1)m^{2k-1}\sum_{d=0}^{r-k}m^d+(m-1)m^{2k-1}\sum_{s=2k-r}^{k-1}\left(\sum_{d=0}^{r-2k+s}m^d\right)\\
&=m^{2k-1}\left(\frac{1}{2}(m-1)\left(\frac{m^{r-k+1}-1}{m-1}\right)+(m-1)\sum_{s=2k-r}^{k-1}\frac{m^{r-2k+s+1}-1}{m-1}\right)\\
&=m^{2k-1}\left(\frac{1}{2}\left(m^{r-k+1}-1\right)+\sum_{i=1}^{r-k}\left(m^i-1\right)\right)\\
&=m^{2k-1}\left(\frac{1}{2}\left(m^{r-k+1}-1\right)+\frac{m^{r-k+1}-m}{m-1}-(r-k)\right)\\
&=m^{2k-1}\left(\frac{1}{2}(m+1)\left(\frac{m^{r-k+1}-1}{m-1}\right)-(r-k+1)\right).
\end{align*}

If $r\ge 2k$, then by Lemma~\ref{pathcountbytypemary}, there are $\frac{1}{2}(m-1)m^{t-1}$ type-$[k,k]$ paths rooted at each vertex $v$ for which $d_T(\rho,v)\le r-k$ and $m^{t}$ type-$[0,t]$ paths rooted at each vertex $v$ for which $d_T(\rho,v)\le r-2k$. Furthermore, for $1\le s\le k-1$, there are $(m-1)m^{t-1}$ type-$[s,t-s]$ paths rooted at each vertex $v$ for which $d_T(\rho,v)\le r-2k+s$. Therefore,
\begin{align*}
P_m(r,t)&=\frac{1}{2}(m-1)m^{2k-1}\sum_{d=0}^{r-k}m^d+(m-1)m^{2k-1}\sum_{s=1}^{k-1}\left(\sum_{d=0}^{r-2k+s}m^d\right)+m^{2k}\sum_{d=0}^{r-2k}m^d\\
&=m^{2k-1}\left(\frac{1}{2}(m-1)\left(\frac{m^{r-k+1}-1}{m-1}\right)+(m-1)\sum_{s=1}^{k-1}\frac{m^{r-2k+s+1}-1}{m-1}+m\left(\frac{m^{r-2k+1}-1}{m-1}\right)\right)\\
&=m^{2k-1}\left(\frac{1}{2}\left(m^{r-k+1}-1\right)+\sum_{i=r-2k+2}^{r-k}\left(m^{i}-1\right)+\frac{m^{r-2k+2}-m}{m-1}\right)\\
&=m^{2k-1}\left(\frac{1}{2}\left(m^{r-k+1}-1\right)+\frac{m^{r-k+1}-m^{r-2k+2}}{m-1}-(k-1)+\frac{m^{r-2k+2}-m}{m-1}\right)\\
&=m^{2k-1}\left(\frac{1}{2}(m+1)\left(\frac{m^{r-k+1}-1}{m-1}\right)-k\right).
\end{align*}
\end{proof}

Theorem~\ref{rootedmary} now follows by combining Propositions~\ref{oddpathsmary} and~\ref{evenpathsmary} with the observation that
\begin{align*}
V_R(d)&=\frac{m^{d+1}-1}{m-1}.
\end{align*}

\section{Perfect unrooted $m$-ary trees}\label{unrooted}

In an unrooted $m$-ary tree $T=(V,E)$, every vertex has degree $1$ or $m+1$. The diameter $D$ of $T$ is the maximum value of $d_T(u,v)$ over all pairs of vertices $u,v\in V$. We call $T$ perfect if and only if every degree-$1$ vertex is distance $D$ from some other vertex.

The symmetry of a perfect unrooted $m$-ary tree $T$ of diameter $D$ depends on whether $D$ is odd or even. We introduce some notation to be used in this respect. If $D=2r-1$ is odd, then $T$ can be constructed by connecting the roots $\rho_1,\rho_2$ of two perfect rooted $m$-ary trees $T_1,T_2$ of depth $r-1$ with an edge $e=\{\rho_1,\rho_2\}$. If $D=2r$ is even, then $T$ can be constructed by connecting the roots $\rho_1,\rho_2$ of the perfect rooted $m$-ary tree $T_1$ of depth $r$ and the perfect rooted $m$-ary tree $T_2$ of depth $r-1$ with an edge $e=\{\rho_1,\rho_2\}$. In either case, a path in $T$ is either contained in $T_1$, contained in $T_2$, or contains $e$.

Let $U_m(D,t)$ denote the number of paths of length $t$ in the perfect unrooted $m$-ary tree of diameter $D$. We consider four cases, depending on the parities of $t$ and $D$. The proofs of the propositions below make repeated use of Lemma~\ref{pathcountbytypemary} and Propositions~\ref{oddpathsmary} and ~\ref{evenpathsmary}.

\begin{proposition}\label{oddDoddpmary}
The number of paths of length $t=2k-1$ in the perfect unrooted $m$-ary tree of diameter $D=2r-1$, where $1\le k\le r$ is
\begin{align*}
U_m(D,t)&=\frac{m^{2k-2}}{m-1}\left(2m^{r-k+1}-(m+1)\right).
\end{align*}
\end{proposition}

\begin{proof}
Let $T$ be the perfect unrooted $m$-ary tree of diameter $D=2r-1$. We use the decomposition of $T$ into $T_1$, $T_2$, and $e$. The number of paths in $T$ of length $t=2k-1$ that contain $e$ is
\begin{align*}
\sum_{s=\max\{0,t-r\}}^{\min\{t-1,r-1\}}m^{t-1}=
\begin{cases}
0,&r<k,\\
(2r-2k+1)m^{2k-2},&k\le r<2k-1\\
(2k-1)m^{2k-2},&r\ge 2k-1.
\end{cases}
\end{align*}

If $r<k$, then $U_m(D,t)=0$. If $r=k$, then the depth of $T_1$ (and of $T_2$) is $r-1<k$, so $P_m(T_1,t)=P_m(T_2,t)=0$. Also, $2r-2k+1=1$, and hence $U_m(D,t)=m^{2k-2}$. If $k<r<2k-1$, then $k\le r-1<2k-2$, and so
\begin{align*}
U_m(D,t)&=2P_m(r-1,2k-1)+(2r-2k+1)m^{2k-2}\\
&=2m^{2k-2}\left(\frac{m^{r-k+1}-1}{m-1}-(r-k+1)\right)+(2r-2k+1)m^{2k-2}\\
&=\frac{m^{2k-2}}{m-1}\left(2m^{r-k+1}-(m+1)\right).
\end{align*}
If $r=2k-1$, then $r-1<2k-1$, and so
\begin{align*}
U_m(D,t)&=2P_m(r-1,2k-1)+(2k-1)m^{2k-2}\\
&=2m^{2k-2}\left(\frac{m^{r-k+1}-1}{m-1}-(r-k+1)\right)+(2k-1)m^{2k-2}\\
&=\frac{m^{2k-2}}{m-1}\left(2m^{r-k+1}-(m+1)\right).
\end{align*}
If $r>2k-1$, then $r-1\ge 2k-1$, and so
\begin{align*}
U_m(D,t)&=2P_m(r-1,2k-1)+(2k-1)m^{2k-2}\\
&=2m^{2k-2}\left(\frac{m^{r-k+1}-1}{m-1}-k\right)+(2k-1)m^{2k-2}\\
&=\frac{m^{2k-2}}{m-1}\left(2m^{r-k+1}-(m+1)\right).
\end{align*}
\end{proof}

\begin{proposition}\label{oddDevenpmary}
The number of paths of length $t=2k$ in the perfect unrooted $m$-ary tree of diameter $D=2r-1$, where $1\le k<r$, is
\begin{align*}
U_m(D,t)&=\frac{m^{2k-1}}{m-1}\left((m+1)m^{r-k}-(m+1)\right).
\end{align*}
\end{proposition}

\begin{proof}
Using the decomposition of $T$ into $T_1$, $T_2$, and $e$, the number of paths of  length $t=2k$ in $T$ that contain $e$ is
\begin{align*}
\sum_{s=\max\{0,t-r\}}^{\min\{t-1,r-1\}}m^{t-1}=
\begin{cases}
0,&r\le k,\\
(2r-2k)m^{2k-1},&k<r\le 2k\\
2k m^{2k-1},&r>2k,
\end{cases}
\end{align*}
and again the result is immediate for $r\le k$. If $k<r\le 2k$, then $k\le r-1<2k$, and so
\begin{align*}
U_m(D,t)&=2P_m(r-1,2k)+(2r-2k)m^{2k-1}\\
&=2m^{2k-1}\left(\frac{1}{2}(m+1)\left(\frac{m^{r-k}-1}{m-1}\right)-(r-k)\right)+(2r-2k)m^{2k-1}\\
&=\frac{m^{2k-1}}{m-1}\left((m+1)m^{r-k}-(m+1)\right).
\end{align*}
If $r>2k$, then $r-1\ge 2k-1$, and so
\begin{align*}
U_m(D,t)&=2P_m(r-1,2k)+2km^{2k-1}\\
&=2m^{2k-1}\left(\frac{1}{2}(m+1)\left(\frac{m^{r-k}-1}{m-1}\right)-k\right)+2km^{2k-1}\\
&=\frac{m^{2k-1}}{m-1}\left((m+1)m^{r-k}-(m+1)\right).
\end{align*}
\end{proof}

\begin{proposition}\label{evenDoddpmary}
The number of paths of length $t=2k-1$ in the perfect unrooted $m$-ary tree of diameter $D=2r$, where $1\le k\le r$, is
\begin{align*}
U_m(D,t)&=\frac{m^{2k-2}}{m-1}\left((m+1)m^{r-k+1}-(m+1)\right).
\end{align*}
\end{proposition}

\begin{proof}
Using the decomposition of $T$ into $T_1$ (depth $r$), $T_2$ (depth $r-1$), and $e$, the number of paths of length $t$ paths in $T$ that contain $e$ is
\begin{align*}
\sum_{s=\max\{0,t-r\}}^{\min\{t-1,r\}}m^{t-1}=
\begin{cases}
0,&r<k,\\
(2r-2k+2)m^{2k-2},&k\le r<2k-1\\
(2k-1)m^{2k-2},&r\ge 2k-1.
\end{cases}
\end{align*}
If $r<k$, then $U_m(D,t)=0$. If $r=k$, then $r-1<k$, and so
\begin{align*}
U_m(D,t)&=P_m(r,2k-1)+P_m(r-1,2k-1)+(2r-2k+2)m^{2k-2}\\
&=m^{2k-2}\left(\frac{m^{r-k+2}-1}{m-1}-(r-k+2)\right)+(2r-2k+2)m^{2k-2}\\
&=\frac{m^{2k-2}}{m-1}\left((m+1)m^{r-k+1}-(m+1)\right).
\end{align*}
If $k<r<2k-1$, then $k\le r-1<2k-2$, and so
\begin{align*}
U_m(D,t)&=P_m(r,2k-1)+P_m(r-1,2k-1)+(2r-2k+2)m^{2k-2}\\
&=m^{2k-2}\left(\frac{m^{r-k+2}-1}{m-1}-(r-k+2)\right)+m^{2k-2}\left(\frac{m^{r-k+1}-1}{m-1}-(r-k+1)\right)\\
&\qquad+(2r-2k+2)m^{2k-2}\\
&=\frac{m^{2k-2}}{m-1}\left((m+1)m^{r-k+1}-(m+1)\right).
\end{align*}
If $r=2k-1$, then $r-1<2k-1$, and so
\begin{align*}
U_m(D,t)&=P_m(r,2k-1)+P_m(r-1,2k-1)+(2k-1)m^{2k-2}\\
&=m^{2k-2}\left(\frac{m^{r-k+2}-1}{m-1}-k\right)+m^{2k-2}\left(\frac{m^{r-k+1}-1}{m-1}-(r-k+1)\right)+(2k-1)m^{2k-2}\\
&=\frac{m^{2k-2}}{m-1}\left((m+1)m^{r-k+1}-(m+1)\right).
\end{align*}
If $r>2k-1$, then $r-1\ge 2k-1$, and so
\begin{align*}
U_m(D,t)&=P_m(r,2k-1)+P_m(r-1,2k-1)+(2k-1)m^{2k-2}\\
&=m^{2k-2}\left(\frac{m^{r-k+2}-1}{m-1}-k\right)+m^{2k-2}\left(\frac{m^{r-k+1}-1}{m-1}-k\right)+(2k-1)m^{2k-2}\\
&=\frac{m^{2k-2}}{m-1}\left((m+1)m^{r-k+1}-(m+1)\right).
\end{align*}
\end{proof}

\begin{proposition}\label{evenDevenpmary}
The number of paths of length $t=2k$ in the perfect unrooted $m$-ary tree of diameter $D=2r$, where $1\le k\le r$, is
\begin{align*}
U_m(D,t)&=
\frac{m^{2k-1}}{m-1}\left(\frac{1}{2}(m+1)^2m^{r-k}-(m+1)\right).
\end{align*}
\end{proposition}

\begin{proof}
Using the decomposition of $T$ into $T_1$ (depth $r$), $T_2$ (depth $r-1$), and $e$, the number of paths of length $t$ in $T$ that contain $e$ is
\begin{align*}
\sum_{s=\max\{0,t-r\}}^{\min\{t-1,r\}}m^{t-1}=
\begin{cases}
0,&r<k,\\
(2r-2k+1)m^{2k-1},&k\le r<2k,\\
2km^{2k-1},&r\ge 2k,
\end{cases}
\end{align*}
and again the result is immediate for $r<k$. If $r=k$, then $r-1<k$, and so
\begin{align*}
U_m(D,t)&=P_m(r,2k)+P_m(r-1,2k)+(2r-2k+1)m^{2k-1}\\
&=m^{2k-1}\left(\frac{1}{2}(m+1)\left(\frac{m^{r-k+1}-1}{m-1}\right)-(r-k+1)\right)+(2r-2k+1)m^{2k-1}\\
&=\frac{m^{2k-1}}{m-1}\left(\frac{1}{2}(m+1)^2m^{r-k}-(m+1)\right).
\end{align*}
If $k<r<2k$, then $k\le r-1<2k$, and so
\begin{align*}
U_m(D,t)&=P_m(r,2k)+P_m(r-1,2k)+(2r-2k+1)m^{2k-1}\\
&=m^{2k-1}\left(\frac{1}{2}(m+1)\left(\frac{m^{r-k+1}-1}{m-1}\right)-(r-k+1)\right)\\
&\qquad+m^{2k-1}\left(\frac{1}{2}(m+1)\left(\frac{m^{r-k}-1}{m-1}\right)-(r-k)\right)+(2r-2k+1)m^{2k-1}\\
&=\frac{m^{2k-1}}{m-1}\left(\frac{1}{2}(m+1)^2m^{r-k}-(m+1)\right).
\end{align*}
If $r=2k$, then $r-1<2k$, and so
\begin{align*}
U_m(D,t)&=P_m(r,2k)+P_m(r-1,2k)+2km^{2k-1}\\
&=m^{2k-1}\left(\frac{1}{2}(m+1)\left(\frac{m^{r-k+1}-1}{m-1}\right)-k\right)\\
&\qquad+m^{2k-1}\left(\frac{1}{2}(m+1)\left(\frac{m^{r-k}-1}{m-1}\right)-(r-k)\right)+2km^{2k-1}\\
&=\frac{m^{2k-1}}{m-1}\left(\frac{1}{2}(m+1)^2m^{r-k}-(m+1)\right).
\end{align*}
If $r>2k$, then $r-1\ge 2k$, and so
\begin{align*}
U_m(D,t)&=P_m(r,2k)+P_m(r-1,2k)+2km^{2k-1}\\
&=m^{2k-1}\left(\frac{1}{2}(m+1)\left(\frac{m^{r-k+1}-1}{m-1}\right)-k\right)\\
&\qquad+m^{2k-1}\left(\frac{1}{2}(m+1)\left(\frac{m^{r-k}-1}{m-1}\right)-k\right)+2km^{2k-1}\\
&=\frac{m^{2k-1}}{m-1}\left(\frac{1}{2}(m+1)^2m^{r-k}-(m+1)\right).
\end{align*}
\end{proof}

Theorem~\ref{unrootedmary} now follows by combining Propositions~\ref{oddDoddpmary} to~\ref{evenDevenpmary} with the observation that
\begin{align*}
V(d)=
\begin{cases}
\frac{2m^{\frac{d+1}{2}}-2}{m-1},&d\ {\rm odd},\\
\frac{(m+1)m^{\frac{d}{2}}-2}{m-1},&d\ {\rm even}.
\end{cases}
\end{align*}

\end{document}